\documentclass{amsart}
\usepackage{amssymb}
\usepackage{amsfonts}

\newtheorem{theorem}{Theorem}
\theoremstyle{plain}

\newtheorem{corollary}{Corollary}

\newtheorem{definition}{Definition}

\numberwithin{equation}{section}
\input{tcilatex}

\begin{document}
\title[Some Integral Inequalities]{SOME NEW INTEGRAL INEQUALITIES FOR
SEVERAL KINDS OF CONVEX FUNCTIONS}
\author{M. Emin \"{O}zdemir$^{\blacksquare }$}
\address{$^{\blacksquare }$Atat\"{u}rk University, K.K. Education Faculty,
Department of Mathematics, 25240, ERZURUM TURKEY}
\email{emos@atauni.edu.tr}
\author{Alper Ekinci$^{\spadesuit }$}
\address{$^{\spadesuit }$A\u{g}r\i\ \.{I}brahim \c{C}e\c{c}en University,
Faculty of Science and Letters, Department of Mathematics, 04100 A\u{G}RI
TURKEY}
\email{alperekinci@hotmail.com}
\author{Ahmet Ocak Akdemir$^{\spadesuit }$}
\email{ahmetakdemir@agri.edu.tr}

\begin{abstract}
In this study, we obtain some new integral inequalities for different
classes of convex functions by using some elementary inequalities and
classical inequalities like general Cauchy inequality and Minkowski
inequality.
\end{abstract}

\maketitle

\section{INTRODUCTION}

The function $f:\left[ a,b\right] \rightarrow 
\mathbb{R}
,$ is said to be convex, if we have%
\begin{equation*}
f\left( tx+\left( 1-t\right) y\right) \leq tf\left( x\right) +\left(
1-t\right) f\left( y\right)
\end{equation*}%
for all $x,y\in \left[ a,b\right] $ and $t\in \left[ 0,1\right] .$ This
definition well-known in the literature and a huge amount of the researchers
interested in this definition. We can define starshaped functions on $\left[
0,b\right] $ which satisfy the condition%
\begin{equation*}
f\left( tx\right) \leq tf\left( x\right)
\end{equation*}
for $t\in \left[ 0,1\right] .$

The concept of $m-$convexity has been introduced by Toader in \cite{TOA}, an
intermediate between the ordinary convexity and starshaped property, as
following:

\begin{definition}
The function $f:\left[ 0,b\right] \rightarrow 
\mathbb{R}
,$ $b>0,$ is said to be $m-$convex, where $m\in \left[ 0,1\right] ,$ if we
have%
\begin{equation*}
f\left( tx+m\left( 1-t\right) y\right) \leq tf\left( x\right) +m\left(
1-t\right) f\left( y\right)
\end{equation*}%
for all $x,y\in \left[ 0,b\right] $ and $t\in \left[ 0,1\right] .$ We say
that $f$ is $m-$concave if $-f$ is $m-$convex.
\end{definition}

Several papers have been written on $m-$convex functions and we refer the
papers \cite{BOP}, \cite{BPR}, \cite{ST}, \cite{MER}, \cite{TOA2} and \cite%
{SS3}.

In \cite{MIH}, Mihe\c{s}an gave definition of $(\alpha ,m)-$convexity as
following;

\begin{definition}
The function $f:[0,b]\rightarrow 
\mathbb{R}
,$ $b>0$ is said to be $(\alpha ,m)-$convex, where $(\alpha ,m)\in \lbrack
0,1]^{2},$ if we have%
\begin{equation*}
f(tx+m(1-t)y)\leq t^{\alpha }f(x)+m(1-t^{\alpha })f(y)
\end{equation*}%
for all $x,y\in \lbrack 0,b]$ and $t\in \lbrack 0,1].$
\end{definition}

Denote by $K_{m}^{\alpha }(b)$ the class of all $(\alpha ,m)-$convex
functions on $[0,b]$ for which $f(0)\leq 0.$ If we choose $(\alpha ,m)=(1,m)$%
, it can be easily seen that $(\alpha ,m)-$convexity reduces to $m-$%
convexity and for $(\alpha ,m)=(1,1),$ we have ordinary convex functions on $%
[0,b].$ In \cite{SET}, Set \textit{et al.} proved some inequalities related
to $(\alpha ,m)-$convex functions.

Following definition of $\log -convexity$ is given by Pe\v{c}ari\'{c} et. al.

\begin{definition}
A function $f:I\rightarrow \left[ 0,\infty \right) $ is said to be
log-convex or multiplicatively convex if log t is convex, or, equivalently,
if for all $x,y\in I$ and $t\in \left[ 0,1\right] $ one has the inequality:%
\begin{equation}
f\left( tx+(1-t)y\right) \leq \left[ f(x)\right] ^{t}\left[ f(y)\right]
^{1-t}.  \label{log}
\end{equation}
\end{definition}

The following inequality which well known in the literature as Minkowski
Inequality is given as;

Let $p\geq 1,$ $0<\dint\limits_{a}^{b}f(x)^{p}dx<\infty ,$ and $%
0<\dint\limits_{a}^{b}g(x)^{p}dx<\infty .$ Then%
\begin{equation}
\left( \dint\limits_{a}^{b}\left( f(x)+g(x)\right) ^{p}dx\right) ^{\frac{1}{p%
}}\leq \left( \dint\limits_{a}^{b}f(x)^{p}dx\right) ^{\frac{1}{p}}+\left(
\dint\limits_{a}^{b}g(x)^{p}dx\right) ^{\frac{1}{p}}.  \label{1.1}
\end{equation}%
The reverse of this inequality was given by Bougoffa in \cite{L}, as
following;

\begin{theorem}
Let $f$ and $g$ be positive functions satisfying%
\begin{equation*}
0<m\leq \frac{f(x)}{g(x)}\leq M,\text{ \ \ \ \ }\forall x\left[ a,b\right] .
\end{equation*}%
Then%
\begin{equation}
\left( \dint\limits_{a}^{b}f(x)^{p}dx\right) ^{\frac{1}{p}}+\left(
\dint\limits_{a}^{b}g(x)^{p}dx\right) ^{\frac{1}{p}}\leq c\left(
\dint\limits_{a}^{b}\left( f(x)+g(x)\right) ^{p}dx\right) ^{\frac{1}{p}}.
\label{m}
\end{equation}%
where $c=\frac{M(m+1)+(M+1)}{(m+1)(M+1)}.$
\end{theorem}

\begin{definition}
\lbrack See \cite{HU}] Let $s\in \left( 0,1\right] .$\ A function $f:\left[
0,\infty \right) \rightarrow \left[ 0,\infty \right) $ is said to be $s-$%
convex in the second sense if \ \ \ \ \ \ \ \ \ \ \ \ 
\begin{equation}
f\left( tx+\left( 1-t\right) y\right) \leq t^{s}f\left( x\right) +\left(
1-t\right) ^{s}f\left( y\right)  \label{1.2}
\end{equation}%
for all $x,y\in 
\mathbb{R}
_{+}$ and $t\in \left[ 0,1\right] .$
\end{definition}

In \cite{WW}, $s-$convexity introduced by Breckner as a generalization of
convex functions. Also, Breckner proved the fact that the set valued map is $%
s-$convex only if the associated support function is $s-$convex function\ in 
\cite{WW2}. Several properties of $s-$convexity in the first sense are
discussed in the paper \cite{HU}. Obviously, $s-$convexity means just
convexity when $s=1$.

\begin{theorem}
\lbrack See \cite{SF}] Suppose that $f:\left[ 0,\infty \right) \rightarrow %
\left[ 0,\infty \right) $ is an $s-$convex function in the second sense,
where $s\in \left( 0,1\right] $ and let $a,b\in \left[ 0,\infty \right) ,$ $%
a<b.$ If $f\in L_{1}\left[ 0,1\right] ,$ then the following inequalities
hold: 
\begin{equation}
2^{s-1}f\left( \frac{a+b}{2}\right) \leq \frac{1}{b-a}\int_{a}^{b}f\left(
x\right) dx\leq \frac{f\left( a\right) +f\left( b\right) }{s+1}.  \label{1.3}
\end{equation}%
The constant $k=\frac{1}{s+1}$ is the best possible in the second inequality
in (\ref{1.3}). The above inequalities are sharp.
\end{theorem}

Some new Hermite-Hadamard type inequalities based on concavity and $s-$%
convexity established by K\i rmac\i\ \textit{et al.} in \cite{UK}. For
related results see the papers \cite{MIQ}, \cite{SF} and \cite{UK}.

In this paper, we prove some inequalities for $m-$convex and $s-$convex and $%
\log -$convex functions and we give some new inequalities for $(\alpha ,m)-$%
convex functions by using some classical inequalities and fairly elementary
analysis.

\section{MAIN RESULTS}

We will start with the following Theorem which is involving $m-$convex
functions.

\begin{theorem}
Suppose that $f,g:\left[ a,b\right] \rightarrow \left[ 0,\infty \right)
,0\leq a<b<\infty ,$ are $m_{1}-$convex and $m_{2}-$convex functions,
respectively, where $m_{1},m_{2}\in \left( 0,1\right] .$ If $f,g\in L_{1}%
\left[ a,b\right] ,$ then the following inequality holds:%
\begin{equation}
\frac{1}{b-a}\int\limits_{a}^{b}f^{\frac{x-a}{b-a}}\left( x\right) g^{\frac{%
b-x}{b-a}}\left( x\right) dx\leq \frac{1}{3}\left[ f(b)+m_{2}g\left( \frac{a%
}{m_{2}}\right) \right] +\frac{1}{6}\left[ g\left( b\right) +m_{1}f\left( 
\frac{a}{m_{1}}\right) \right] .  \label{z1}
\end{equation}
\end{theorem}

\begin{proof}
From $m_{1}-$convexity and $m_{2}-$convexity of $f$ and $g$ respectively, we
can write%
\begin{equation*}
f^{t}\left( tb+\left( 1-t\right) a\right) \leq \left[ tf\left( b\right)
+m_{1}\left( 1-t\right) f\left( \frac{a}{m_{1}}\right) \right] ^{t}
\end{equation*}%
and%
\begin{equation*}
g^{(1-t)}\left( tb+\left( 1-t\right) a\right) \leq \left[ tg\left( b\right)
+m_{2}\left( 1-t\right) g\left( \frac{a}{m_{2}}\right) \right] ^{(1-t)}.
\end{equation*}%
Since $f,g$ are non-negative, we have%
\begin{eqnarray}
&&f^{t}\left( tb+\left( 1-t\right) a\right) g^{(1-t)}\left( tb+\left(
1-t\right) a\right)   \label{1} \\
&\leq &\left[ tf\left( b\right) +m_{1}\left( 1-t\right) f\left( \frac{a}{%
m_{1}}\right) \right] ^{t}\left[ tg\left( b\right) +m_{2}\left( 1-t\right)
g\left( \frac{a}{m_{2}}\right) \right] ^{(1-t)}.  \notag
\end{eqnarray}%
Recall the General Cauchy Inequality (see \cite{T}, Theorem 3.1), let $%
\alpha $ and $\beta $ be positive real numbers satisfying $\alpha +\beta =1$%
. Then for every positive real numbers $x$ and $y$, we always have%
\begin{equation*}
\alpha x+\beta y\geq x^{\alpha }y^{\beta }.
\end{equation*}%
Applying the General Cauchy Inequality to the right hand side of (\ref{1}),
we get%
\begin{eqnarray*}
&&f^{t}\left( tb+\left( 1-t\right) a\right) g^{(1-t)}\left( tb+\left(
1-t\right) a\right)  \\
&\leq &t\left[ tf\left( b\right) +m_{1}\left( 1-t\right) f\left( \frac{a}{%
m_{1}}\right) \right] +\left( 1-t\right) \left[ tg\left( b\right)
+m_{2}\left( 1-t\right) g\left( \frac{a}{m_{2}}\right) \right] .
\end{eqnarray*}%
By integrating with respect to $t$ over $\left[ 0,1\right] ,$ we have%
\begin{eqnarray*}
&&\int\limits_{0}^{1}f^{t}\left( tb+\left( 1-t\right) a\right)
g^{(1-t)}\left( tb+\left( 1-t\right) a\right) dt \\
&\leq &\frac{1}{3}\left[ f(b)+m_{2}g\left( \frac{a}{m_{2}}\right) \right] +%
\frac{1}{6}\left[ g\left( b\right) +m_{1}f\left( \frac{a}{m_{1}}\right) %
\right] .
\end{eqnarray*}%
Hence, by taking into account the change of the variable $tb+\left(
1-t\right) a=x,$ $(b-a)dt=dx,$ we obtain the required result.
\end{proof}

\begin{corollary}
If we choose $m_{1}=m_{2}=1$ in Theorem 3, we have the inequality;%
\begin{equation*}
\frac{1}{b-a}\int\limits_{a}^{b}f^{\frac{x-a}{b-a}}\left( x\right) g^{\frac{%
b-x}{b-a}}\left( x\right) dx\leq \frac{1}{3}\left[ f(b)+g\left( a\right) %
\right] +\frac{1}{6}\left[ g\left( b\right) +f\left( a\right) \right] .
\end{equation*}
\end{corollary}

Another result for $m-$convex functions is emboided in the following Theorem.

\begin{theorem}
Suppose that $f,g:\left[ 0,b\right] \rightarrow 
\mathbb{R}
,$ $b>0,$ are $m_{1}-$convex and $m_{2}-$convex functions, respectively,
where $m_{1},m_{2}\in \left( 0,1\right] .$ If $f\in L_{1}\left[ a,b\right] ,$
then the following inequality holds:%
\begin{eqnarray}
&&\frac{g\left( b\right) }{\left( b-a\right) ^{2}}\int\limits_{a}^{b}(x-a)f%
\left( x\right) dx+m_{2}\frac{g\left( \frac{a}{m_{2}}\right) }{\left(
b-a\right) ^{2}}\int\limits_{a}^{b}(b-x)f\left( x\right) dx  \label{z2} \\
&&+\frac{f\left( b\right) }{\left( b-a\right) ^{2}}\int\limits_{a}^{b}(x-a)g%
\left( x\right) dx+m_{1}\frac{f\left( \frac{a}{m_{1}}\right) }{\left(
b-a\right) ^{2}}\int\limits_{a}^{b}(b-x)g\left( x\right) dx  \notag \\
&\leq &\frac{1}{b-a}\int\limits_{a}^{b}f\left( x\right) g\left( x\right) dx+%
\frac{1}{3}f\left( b\right) g\left( b\right) +\frac{m_{1}}{6}f\left( \frac{a%
}{m_{1}}\right) g\left( b\right)  \notag \\
&&+\frac{m_{2}}{6}f\left( b\right) g\left( \frac{a}{m_{2}}\right) +\frac{%
m_{1}m_{2}}{3}f\left( \frac{a}{m_{1}}\right) g\left( \frac{a}{m_{2}}\right) .
\notag
\end{eqnarray}
\end{theorem}

\begin{proof}
Since $f$ and $g$ are $m_{1}-$convex and $m_{2}-$convex functions,
respectively, we can write%
\begin{equation*}
f\left( tb+\left( 1-t\right) a\right) \leq tf\left( b\right) +m_{1}\left(
1-t\right) f\left( \frac{a}{m_{1}}\right) 
\end{equation*}%
and%
\begin{equation*}
g\left( tb+\left( 1-t\right) a\right) \leq tg\left( b\right) +m_{2}\left(
1-t\right) g\left( \frac{a}{m_{2}}\right) .
\end{equation*}%
By applying the elementary inequality, $e\leq f$ and $p\leq r$, then $%
er+fp\leq ep+fr$ for $e,f,p,r\in 
\mathbb{R}
,$ to above inequalities, we get:%
\begin{eqnarray*}
&&f\left( tb+\left( 1-t\right) a\right) \left[ tg\left( b\right)
+m_{2}\left( 1-t\right) g\left( \frac{a}{m_{2}}\right) \right]  \\
&&+g\left( tb+\left( 1-t\right) a\right) \left[ tf\left( b\right)
+m_{1}\left( 1-t\right) f\left( \frac{a}{m_{1}}\right) \right]  \\
&\leq &f\left( tb+\left( 1-t\right) a\right) g\left( tb+\left( 1-t\right)
a\right)  \\
&&+\left[ tg\left( b\right) +m_{2}\left( 1-t\right) g\left( \frac{a}{m_{2}}%
\right) \right] \left[ tf\left( b\right) +m_{1}\left( 1-t\right) f\left( 
\frac{a}{m_{1}}\right) \right] .
\end{eqnarray*}%
With a simple computation, we obtain%
\begin{eqnarray*}
&&tf\left( tb+\left( 1-t\right) a\right) g\left( b\right) +m_{2}\left(
1-t\right) f\left( tb+\left( 1-t\right) a\right) g\left( \frac{a}{m_{2}}%
\right)  \\
&&+tf\left( b\right) g\left( tb+\left( 1-t\right) a\right) +m_{1}\left(
1-t\right) f\left( \frac{a}{m_{1}}\right) g\left( tb+\left( 1-t\right)
a\right)  \\
&\leq &f\left( tb+\left( 1-t\right) a\right) g\left( tb+\left( 1-t\right)
a\right) +t^{2}f\left( b\right) g\left( b\right) +m_{1}t\left( 1-t\right)
f\left( \frac{a}{m_{1}}\right) g\left( b\right)  \\
&&+m_{2}t\left( 1-t\right) f\left( b\right) g\left( \frac{a}{m_{2}}\right)
+m_{1}m_{2}\left( 1-t\right) ^{2}f\left( \frac{a}{m_{1}}\right) g\left( 
\frac{a}{m_{2}}\right) .
\end{eqnarray*}%
By integrating this inequality with respect to $t$ over $\left[ 0,1\right] $
and by using the change of the variable $tb+\left( 1-t\right) a=x,$ $%
(b-a)dt=dx,$ the proof is completed.
\end{proof}

\begin{corollary}
If we choose $m_{1}=m_{2}=1$ in Theorem 4, we have the inequality;%
\begin{eqnarray*}
&&\frac{g\left( b\right) }{\left( b-a\right) ^{2}}\int\limits_{a}^{b}(x-a)f%
\left( x\right) dx+\frac{g\left( a\right) }{\left( b-a\right) ^{2}}%
\int\limits_{a}^{b}(b-x)f\left( x\right) dx \\
&&+\frac{f\left( b\right) }{\left( b-a\right) ^{2}}\int\limits_{a}^{b}(x-a)g%
\left( x\right) dx+\frac{f\left( a\right) }{\left( b-a\right) ^{2}}%
\int\limits_{a}^{b}(b-x)g\left( x\right) dx \\
&\leq &\frac{1}{b-a}\int\limits_{a}^{b}f\left( x\right) g\left( x\right) dx+%
\frac{1}{3}M(a,b)+\frac{1}{6}N(a,b)
\end{eqnarray*}%
where $M\left( a,b\right) =f\left( a\right) g\left( a\right) +f\left(
b\right) g\left( b\right) $ and $N\left( a,b\right) =f\left( a\right)
g\left( b\right) +f\left( b\right) g\left( a\right) .$
\end{corollary}

\begin{corollary}
If we choose the functions $f,$ $g$ as increasing functions in Corollary 2,
we obtain the following result:%
\begin{eqnarray*}
&&\frac{g\left( a\right) }{\left( b-a\right) ^{2}}\left[ \int%
\limits_{a}^{b}(x-a)f\left( x\right) dx+\int\limits_{a}^{b}(b-x)f\left(
x\right) dx\right] \\
&&+\frac{f\left( a\right) }{\left( b-a\right) ^{2}}\left[ \int%
\limits_{a}^{b}(x-a)g\left( x\right) dx+\int\limits_{a}^{b}(b-x)g\left(
x\right) dx\right] \\
&\leq &\frac{1}{b-a}\int\limits_{a}^{b}f\left( x\right) g\left( x\right) dx+%
\frac{1}{3}M(a,b)+\frac{1}{6}N(a,b).
\end{eqnarray*}
\end{corollary}

\begin{corollary}
If we chose $m_{1}=m_{2}=1$ \ and $g\left( x\right) =1$ in Theorem 4, we
have the inequality;%
\begin{eqnarray*}
&&\frac{1}{\left( b-a\right) ^{2}}\left[ \int\limits_{a}^{b}(x-a)f\left(
x\right) dx+\int\limits_{a}^{b}(b-x)f\left( x\right) dx\right] \\
&&+\frac{f\left( b\right) }{\left( b-a\right) ^{2}}\int%
\limits_{a}^{b}(x-a)dx+\frac{f\left( a\right) }{\left( b-a\right) ^{2}}%
\int\limits_{a}^{b}(b-x)dx \\
&\leq &\frac{1}{b-a}\int\limits_{a}^{b}f\left( x\right) dx+\frac{f\left(
a\right) +f\left( b\right) }{2}.
\end{eqnarray*}%
Following inequality also holds for $m-$convex functions.
\end{corollary}

\begin{theorem}
Suppose that $f,g:\left[ a,b\right] \rightarrow \left[ 0,\infty \right)
,0\leq a<b<\infty ,$ are $m_{1}-$convex and $m_{2}-$convex functions,
respectively, where $m_{1},m_{2}\in \left( 0,1\right] .$ If $f,g\in L_{1}%
\left[ a,b\right] $ and $f,g$ satisfy following condition%
\begin{equation*}
0<m\leq \frac{f(x)}{g(x)}\leq M,\text{ \ \ \ \ }\forall x\left[ a,b\right]
\end{equation*}%
then the following inequality holds:%
\begin{eqnarray*}
&&\frac{1}{c}\left[ \left( \dint\limits_{a}^{b}f(x)^{p}dx\right) ^{\frac{1}{p%
}}+\left( \dint\limits_{a}^{b}g(x)^{p}dx\right) ^{\frac{1}{p}}\right] \\
&\leq &\left( \frac{2^{p-1}\left( b-a\right) }{p+1}\right) ^{\frac{1}{p}%
}\left( \left[ f\left( b\right) +g\left( b\right) \right] ^{p}-\left[
m_{1}f\left( \frac{a}{m_{1}}\right) +m_{2}g\left( \frac{a}{m_{2}}\right) %
\right] ^{p}\right) ^{\frac{1}{p}}
\end{eqnarray*}%
where $c=\frac{M(m+1)+(M+1)}{(m+1)(M+1)}$ and $p\geq 1.$
\end{theorem}

\begin{proof}
Since $f$ and $g$ are $m_{1}-$convex and $m_{2}-$convex functions,
respectively, we can write%
\begin{equation}
f\left( tb+\left( 1-t\right) a\right) \leq tf\left( b\right) +m_{1}\left(
1-t\right) f\left( \frac{a}{m_{1}}\right)  \label{a}
\end{equation}%
and%
\begin{equation}
g\left( tb+\left( 1-t\right) a\right) \leq tg\left( b\right) +m_{2}\left(
1-t\right) g\left( \frac{a}{m_{2}}\right) .  \label{b}
\end{equation}%
By adding (\ref{a}) and (\ref{b}), we get 
\begin{eqnarray}
f\left( tb+\left( 1-t\right) a\right) +g\left( tb+\left( 1-t\right) a\right)
&\leq &tf\left( b\right) +m_{1}\left( 1-t\right) f\left( \frac{a}{m_{1}}%
\right)  \notag \\
&&+tg\left( b\right) +m_{2}\left( 1-t\right) g\left( \frac{a}{m_{2}}\right) .
\label{c}
\end{eqnarray}%
For $p\geq 1,$ taking $p-$th power of both sides of \ the inequality (\ref{c}%
) and by using the elementary inequality, $(e+f)^{p}\leq 2^{p-1}\left(
e^{p}+f^{p}\right) $ where $e,f\in 
\mathbb{R}
,$ then we get%
\begin{eqnarray*}
&&\left[ f\left( tb+\left( 1-t\right) a\right) +g\left( tb+\left( 1-t\right)
a\right) \right] ^{p} \\
&\leq &2^{p-1}\left( t^{p}\left[ f\left( b\right) +g\left( b\right) \right]
^{p}+\left( 1-t\right) ^{p}\left[ m_{1}f\left( \frac{a}{m_{1}}\right)
+m_{2}g\left( \frac{a}{m_{2}}\right) \right] ^{p}\right) .
\end{eqnarray*}%
Integrating with respect to $t$ over $\left[ 0,1\right] $ and by using the
change of the variable $tb+\left( 1-t\right) a=x$ and $(b-a)dt=dx,$ we obtain%
\begin{equation}
\frac{1}{b-a}\dint\limits_{a}^{b}\left( f(x)+g(x)\right) ^{p}dx\leq \frac{%
2^{p-1}}{p+1}\left( \left[ f\left( b\right) +g\left( b\right) \right] ^{p}-%
\left[ m_{1}f\left( \frac{a}{m_{1}}\right) +m_{2}g\left( \frac{a}{m_{2}}%
\right) \right] ^{p}\right) .  \label{d}
\end{equation}%
By taking $\frac{1}{p}-$th power of both sides of \ the inequality (\ref{d})
and by using the inequality (\ref{m}), we get the desired inequality. Which
completes the proof.
\end{proof}

\begin{corollary}
Under the assumptions of Theorem 5, if we choose $m_{1}=m_{2}=1$ and take
the limit of both sides as $p\rightarrow 1,$ we obtain the following
inequality:
\end{corollary}

\begin{equation*}
\int\limits_{a}^{b}\left[ f\left( x\right) +g\left( x\right) \right] dx\leq
\left( \frac{c\left( b-a\right) }{2}\right) \left[ \left[ f\left( b\right)
+g\left( b\right) \right] -\left[ \left( f\left( a\right) +g\left( a\right)
\right) \right] \right]
\end{equation*}

We will give an inequality for $s-$convex functions in the following
theorem. In the next theorem we will also make use of the Beta function of
Euler type, which is for $x,y>0$ defined

as%
\begin{equation*}
\beta (x,y)=\int\limits_{0}^{1}t^{x-1}(1-t)^{y-1}dt.
\end{equation*}

\begin{theorem}
Suppose that $f,g:\left[ 0,\infty \right) \rightarrow \left[ 0,\infty
\right) $ are $s_{1}-$convex and $s_{2}-$convex functions in the second
sense, respectively, where $s_{1},s_{2}\in \left[ 0,1\right] .$ Then the
following inequality holds:%
\begin{eqnarray*}
\frac{1}{b-a}\int\limits_{a}^{b}f^{\frac{x-a}{b-a}}\left( x\right) g^{\frac{%
b-x}{b-a}}\left( x\right) dx &\leq &\frac{1}{s_{1}+2}f(b)+\beta \left(
2,s_{1}+1\right) f\left( a\right) \\
&&+\frac{1}{s_{2}+2}g\left( b\right) +\beta \left( 2,s_{2}+1\right) g\left(
a\right) .
\end{eqnarray*}
\end{theorem}

\begin{proof}
Since $f$ and $g$ are $s_{1}-$convex and $s_{2}-$convex functions,
respectively, we can write%
\begin{equation*}
f^{t}\left( tb+\left( 1-t\right) a\right) \leq \left[ t^{s_{1}}f\left(
b\right) +\left( 1-t\right) ^{s_{1}}f\left( a\right) \right] ^{t}
\end{equation*}%
and%
\begin{equation*}
g^{\left( 1-t\right) }\left( tb+\left( 1-t\right) a\right) \leq \left[
t^{s_{2}}g\left( b\right) +\left( 1-t\right) ^{s_{2}}g\left( a\right) \right]
^{\left( 1-t\right) }.
\end{equation*}%
Since $f,g$ are non-negative, we have%
\begin{eqnarray}
&&f^{t}\left( tb+\left( 1-t\right) a\right) g^{(1-t)}\left( tb+\left(
1-t\right) a\right)  \label{k} \\
&\leq &\left[ t^{s_{1}}f\left( b\right) +\left( 1-t\right) ^{s_{1}}f\left(
a\right) \right] ^{t}\left[ t^{s_{2}}g\left( b\right) +\left( 1-t\right)
^{s_{2}}g\left( a\right) \right] ^{\left( 1-t\right) }.  \notag
\end{eqnarray}%
By using the General Cauchy Inequality in (\ref{k}), we get%
\begin{eqnarray*}
&&f^{t}\left( tb+\left( 1-t\right) a\right) g^{(1-t)}\left( tb+\left(
1-t\right) a\right) \\
&\leq &t\left[ t^{s_{1}}f\left( b\right) +\left( 1-t\right) ^{s_{1}}f\left(
a\right) \right] +\left( 1-t\right) \left[ t^{s_{2}}g\left( b\right) +\left(
1-t\right) ^{s_{2}}g\left( a\right) \right] .
\end{eqnarray*}%
By integrating with respect to $t$ over $\left[ 0,1\right] ,$ we have%
\begin{eqnarray*}
&&\int\limits_{0}^{1}f^{t}\left( tb+\left( 1-t\right) a\right)
g^{(1-t)}\left( tb+\left( 1-t\right) a\right) dt \\
&\leq &\int\limits_{0}^{1}\left[ t^{s_{1}+1}f(b)+t\left( 1-t\right)
^{s_{1}}f\left( a\right) +t^{s_{2}+1}g\left( b\right) +t\left( 1-t\right)
^{s_{2}}g\left( b\right) \right] dt.
\end{eqnarray*}%
Hence, by taking into account the change of the variable $tb+\left(
1-t\right) a=x,$ $(b-a)dt=dx,$ we obtain the required result.
\end{proof}

\begin{corollary}
If we choose $s_{1}=s_{2}=1$ in Theorem 6, we have the inequality;%
\begin{equation*}
\frac{1}{b-a}\int\limits_{a}^{b}f^{\frac{x-a}{b-a}}\left( x\right) g^{\frac{%
b-x}{b-a}}\left( x\right) dx\leq \frac{1}{3}\left[ f(b)+g\left( b\right) %
\right] +\frac{1}{6}\left[ f\left( a\right) +g\left( a\right) \right] .
\end{equation*}
\end{corollary}

\begin{theorem}
Let $f,g$ be $s-convex$ functions in the second sense and $\alpha +\beta =1$
then the following inequality holds:%
\begin{equation*}
\frac{1}{b-a}\int\limits_{a}^{b}f^{\alpha }\left( x\right) .g^{\beta }\left(
x\right) dx\leq \frac{1}{s+1}\left[ \alpha \left[ f\left( a\right) +f\left(
b\right) \right] +\beta \left[ g\left( a\right) +g\left( b\right) \right] %
\right] .
\end{equation*}
\end{theorem}

\begin{proof}
If we use the general Cauchy inaquality with $s-$convexity of $f$ and $g$ we
get:

\begin{eqnarray*}
&&f^{\alpha }\left( ta+\left( 1-t\right) b\right) g^{\beta }\left( ta+\left(
1-t\right) b\right) \\
&\leq &\alpha f\left( ta+\left( 1-t\right) b\right) +\beta g\left( ta+\left(
1-t\right) b\right) \\
&\leq &\alpha \left[ t^{s}f\left( a\right) +\left( 1-t\right) ^{s}f\left(
b\right) \right] +\beta \left[ t^{s}g\left( a\right) +\left( 1-t\right)
^{s}g\left( b\right) \right]
\end{eqnarray*}

By integrating with respect to $t$ over $\left[ 0,1\right] ,$ we have%
\begin{eqnarray*}
&&\frac{1}{b-a}\int\limits_{a}^{b}f^{\alpha }\left( ta+\left( 1-t\right)
b\right) g^{\beta }\left( ta+\left( 1-t\right) b\right) dt \\
&\leq &\alpha \int\limits_{a}^{b}\left[ t^{s}f\left( a\right) +\left(
1-t\right) ^{s}f\left( b\right) \right] dt+\beta \int\limits_{a}^{b}\left[
t^{s}g\left( a\right) +\left( 1-t\right) ^{s}g\left( b\right) \right] dt \\
&=&\frac{1}{s+1}\left[ \alpha \left[ f\left( a\right) +f\left( b\right) %
\right] +\beta \left[ g\left( a\right) +g\left( b\right) \right] \right] .
\end{eqnarray*}%
With the change of variable $ta+\left( 1-t\right) b=x$ we obtain:%
\begin{equation*}
\frac{1}{b-a}\int\limits_{a}^{b}f^{\alpha }\left( x\right) g^{\beta }\left(
x\right) dx\leq \frac{1}{s+1}\left[ \alpha \left[ f\left( a\right) +f\left(
b\right) \right] +\beta \left[ g\left( a\right) +g\left( b\right) \right] %
\right] .
\end{equation*}%
That is the desired result.

A similar result for $\log -convex$ functions is as follows:
\end{proof}

\begin{theorem}
Let $f,g$ be $\log -convex$ functions and $\alpha +\beta =1$ where $L$
denotes the logarithmic mean then the following inequality holds:%
\begin{equation*}
\frac{1}{b-a}\int\limits_{a}^{b}f^{\alpha }\left( x\right) g^{\beta }\left(
x\right) dx\leq \alpha L\left[ f\left( a\right) ,f\left( b\right) \right]
+\beta L\left[ g\left( a\right) ,g\left( b\right) \right] .
\end{equation*}%
Logarithmic Mean: $L(a,b)=\frac{a-b}{\log \left( a\right) -\log \left(
b\right) }$ where $\ a,b\in 
\mathbb{R}
^{+}.$
\end{theorem}

\begin{proof}
If we use the general Cauchy inaquality with $\log -$convexity of $f$ and $g$
we get:

\begin{eqnarray*}
&&f^{\alpha }\left( ta+\left( 1-t\right) b\right) g^{\beta }\left( ta+\left(
1-t\right) b\right) \\
&\leq &\alpha f\left( ta+\left( 1-t\right) b\right) +\beta g\left( ta+\left(
1-t\right) b\right) \\
&\leq &\alpha \left[ f\left( a\right) \right] ^{t}\left[ f\left( b\right) %
\right] ^{\left( 1-t\right) }+\beta \left[ g\left( a\right) \right] ^{t}%
\left[ g\left( b\right) \right] ^{\left( 1-t\right) }
\end{eqnarray*}

By integrating both sides with respect to $t$ over $\left[ 0,1\right] ,$ we
have%
\begin{eqnarray*}
&&\frac{1}{b-a}\int\limits_{a}^{b}f^{\alpha }\left( ta+\left( 1-t\right)
b\right) g^{\beta }\left( ta+\left( 1-t\right) b\right) dt \\
&\leq &\alpha \int\limits_{a}^{b}\left[ f\left( a\right) \right] ^{t}\left[
f\left( b\right) \right] ^{\left( 1-t\right) }dt+\beta \int\limits_{a}^{b}%
\left[ g\left( a\right) \right] ^{t}\left[ g\left( b\right) \right] ^{\left(
1-t\right) }dt \\
&=&\alpha \frac{f\left( a\right) -f\left( b\right) }{\log \left[ f\left(
a\right) \right] -\log \left[ f\left( b\right) \right] }+\beta \frac{g\left(
a\right) -g\left( b\right) }{\log \left[ g\left( a\right) \right] -\log %
\left[ g\left( b\right) \right] } \\
&=&\alpha L\left[ f\left( a\right) ,f\left( b\right) \right] +\beta L\left[
g\left( a\right) ,g\left( b\right) \right] .
\end{eqnarray*}%
With the change of variable $ta+\left( 1-t\right) b=x$ we obtain the desired
result.

In following two Theorems we obtain results for $\left( \alpha ,m\right)
-convex$ functions:
\end{proof}

\begin{theorem}
Suppose that $f,g:\left[ a,b\right] \rightarrow \left[ 0,\infty \right) ,$ $%
0\leq a<b<\infty ,$ are $(\alpha _{1},m_{1})-$convex and $(\alpha
_{2},m_{2})-$convex functions, respectively, where $\alpha _{1},m_{1},\alpha
_{2},m_{2}\in (0,1].$ If $f,g\in L_{1}\left[ a,b\right] ,$ then the
following inequality holds:%
\begin{eqnarray*}
&&\frac{1}{b-a}\int\limits_{a}^{b}f^{\frac{x-a}{b-a}}\left( x\right) g^{%
\frac{b-x}{b-a}}\left( x\right) dx \\
&\leq &\frac{1}{\alpha _{1}+2}f\left( b\right) +\frac{m_{1}}{2\left( \alpha
_{1}+2\right) }f\left( \frac{a}{m_{1}}\right) \\
&&+\frac{1}{\left( \alpha _{2}+1\right) \left( \alpha _{2}+2\right) }g\left(
b\right) +\frac{m_{2}\left( \alpha _{2}^{2}+3\alpha \right) }{2\left( \alpha
_{2}+1\right) \left( \alpha _{2}+2\right) }g\left( \frac{a}{m_{2}}\right) .
\end{eqnarray*}
\end{theorem}

\begin{proof}
Since $f$, $g$ are $(\alpha _{1},m_{1})-$convex and $(\alpha _{2},m_{2})-$%
convex functions, respectively, we can write%
\begin{equation*}
f^{t}\left( tb+\left( 1-t\right) a\right) \leq \left[ t^{\alpha _{1}}f\left(
b\right) +m_{1}\left( 1-t^{\alpha _{1}}\right) f\left( \frac{a}{m_{1}}%
\right) \right] ^{t}
\end{equation*}%
and%
\begin{equation*}
g^{\left( 1-t\right) }\left( tb+\left( 1-t\right) a\right) \leq \left[
t^{\alpha _{2}}g\left( b\right) +m_{2}\left( 1-t^{\alpha _{2}}\right)
g\left( \frac{a}{m_{2}}\right) \right] ^{\left( 1-t\right) }.
\end{equation*}%
Since $f,g$ are non-negative, we have%
\begin{eqnarray}
&&f^{t}\left( tb+\left( 1-t\right) a\right) g^{(1-t)}\left( tb+\left(
1-t\right) a\right)  \label{n} \\
&\leq &\left[ t^{\alpha _{1}}f\left( b\right) +m_{1}\left( 1-t^{\alpha
_{1}}\right) f\left( \frac{a}{m_{1}}\right) \right] ^{t}\left[ t^{\alpha
_{2}}g\left( b\right) +m_{2}\left( 1-t^{\alpha _{2}}\right) g\left( \frac{a}{%
m_{2}}\right) \right] ^{\left( 1-t\right) }.  \notag
\end{eqnarray}%
By using the General Cauchy Inequality in (\ref{n}), we get%
\begin{eqnarray*}
&&f^{t}\left( tb+\left( 1-t\right) a\right) g^{(1-t)}\left( tb+\left(
1-t\right) a\right) \\
&\leq &t\left[ t^{\alpha _{1}}f\left( b\right) +m_{1}\left( 1-t^{\alpha
_{1}}\right) f\left( \frac{a}{m_{1}}\right) \right] +\left( 1-t\right) \left[
t^{\alpha _{2}}g\left( b\right) +m_{2}\left( 1-t^{\alpha _{2}}\right)
g\left( \frac{a}{m_{2}}\right) \right] .
\end{eqnarray*}%
By integrating with respect to $t$ over $\left[ 0,1\right] ,$ we have%
\begin{eqnarray*}
&&\int\limits_{0}^{1}f^{t}\left( tb+\left( 1-t\right) a\right)
g^{(1-t)}\left( tb+\left( 1-t\right) a\right) dt \\
&\leq &\frac{1}{\alpha _{1}+2}f\left( b\right) +\frac{m_{1}}{2\left( \alpha
_{1}+2\right) }f\left( \frac{a}{m_{1}}\right) \\
&&+\frac{1}{\left( \alpha _{2}+1\right) \left( \alpha _{2}+2\right) }g\left(
b\right) +\frac{m_{2}\left( \alpha _{2}^{2}+3\alpha \right) }{2\left( \alpha
_{2}+1\right) \left( \alpha _{2}+2\right) }g\left( \frac{a}{m_{2}}\right) .
\end{eqnarray*}%
Hence, by taking into account the change of the variable $tb+\left(
1-t\right) a=x,$ $(b-a)dt=dx,$ we obtain the required result.
\end{proof}

\begin{corollary}
If we choose $\alpha _{1}=\alpha _{2}=1$ in Theorem 9, we have the
inequality (\ref{z1}).
\end{corollary}

\begin{theorem}
Suppose that $f,g:\left[ a,b\right] \rightarrow \left[ 0,\infty \right) ,$ $%
0\leq a<b<\infty ,$ are $(\alpha _{1},m_{1})-$convex and $(\alpha
_{2},m_{2})-$convex functions, respectively, where $\alpha _{1},m_{1},\alpha
_{2},m_{2}\in (0,1].$ If $f,g\in L_{1}\left[ a,b\right] ,$ then the
following inequality holds:%
\begin{eqnarray*}
&&\frac{g\left( b\right) }{\left( b-a\right) ^{\alpha _{2}+1}}%
\int\limits_{a}^{b}(x-a)^{\alpha _{2}}f\left( x\right) dx+m_{2}\frac{g\left( 
\frac{a}{m_{2}}\right) }{\left( b-a\right) ^{\alpha _{2}+1}}%
\int\limits_{a}^{b}\left[ (b-a)^{\alpha _{2}}-(x-a)^{\alpha _{2}}\right]
f\left( x\right) dx \\
&&+\frac{f\left( b\right) }{\left( b-a\right) ^{\alpha _{1}+1}}%
\int\limits_{a}^{b}(x-a)^{\alpha _{1}}g\left( x\right) dx+m_{1}\frac{f\left( 
\frac{a}{m_{1}}\right) }{\left( b-a\right) ^{\alpha _{1}+1}}%
\int\limits_{a}^{b}\left[ (b-a)^{\alpha _{1}}-(x-a)^{\alpha _{1}}\right]
g\left( x\right) dx \\
&\leq &\frac{1}{b-a}\int\limits_{a}^{b}f\left( x\right) g\left( x\right) dx+%
\frac{1}{\alpha _{1}+\alpha _{2}+1}f\left( b\right) g\left( b\right) +\frac{%
m_{2}\alpha _{2}}{\left( \alpha _{1}+1\right) \left( \alpha _{1}+\alpha
_{2}+1\right) }g\left( \frac{a}{m_{2}}\right) f\left( b\right) \\
&&+\frac{m_{1}\alpha _{1}}{\left( \alpha _{1}+1\right) \left( \alpha
_{1}+\alpha _{2}+1\right) }f\left( \frac{a}{m_{1}}\right) g\left( b\right) +%
\frac{m_{1}m_{2}\alpha _{1}\alpha _{2}\left( \alpha _{1}+\alpha
_{2}+2\right) }{\left( \alpha _{1}+1\right) \left( \alpha _{2}+1\right)
\left( \alpha _{1}+\alpha _{2}+1\right) }f\left( \frac{a}{m_{1}}\right)
g\left( \frac{a}{m_{2}}\right) .
\end{eqnarray*}
\end{theorem}

\begin{proof}
Since $f$, $g$ are $(\alpha _{1},m_{1})-$convex and $(\alpha _{2},m_{2})-$%
convex functions, respectively, we can write%
\begin{equation*}
f\left( tb+\left( 1-t\right) a\right) \leq t^{\alpha _{1}}f\left( b\right)
+m_{1}\left( 1-t^{\alpha _{1}}\right) f\left( \frac{a}{m_{1}}\right)
\end{equation*}%
and%
\begin{equation*}
g\left( tb+\left( 1-t\right) a\right) \leq t^{\alpha _{2}}g\left( b\right)
+m_{2}\left( 1-t^{\alpha _{2}}\right) g\left( \frac{a}{m_{2}}\right) .
\end{equation*}%
By using the elementary inequality, $e\leq f$ and $p\leq r$, then $er+fp\leq
ep+fr$ for $e,f,p,r\in 
\mathbb{R}
$ and by a similar argument to the proof of Theorem 4, we get the required
result.
\end{proof}

\begin{corollary}
If we choose $\alpha _{1}=\alpha _{2}=1$ in Theorem 10, we have the
inequality (\ref{z2}).
\end{corollary}

\end{document}